\documentclass[a4paper,12pt,twoside]{article}

\usepackage{amsmath,amsthm,amscd,amsfonts,amssymb,amsxtra}
\usepackage[hmargin=1.2in,vmargin=1.2in]{geometry}
\usepackage[utf8]{inputenc}
\usepackage{graphicx}
\usepackage{microtype}
\usepackage[pdftex,bookmarks,colorlinks=false]{hyperref}
\usepackage{booktabs,enumitem,caption,subcaption}
\usepackage[mathscr]{euscript}

\usepackage[auth-sc-lg,affil-it]{authblk}
\usepackage{setspace}
\numberwithin{equation}{section}

\newtheorem{thm}{Theorem}[section]
\newtheorem{defn}{Definition}[section]

\newtheorem{prop}[thm]{Proposition}

\newtheorem{ill}{Illustration}

\def\ni{\noindent}
\def\N{\mathbb{N}}

\def\cP{\mathcal{P}}

\title{\textbf{\sc A Note on the Tattoo Index of Graphs}}

\author{Johan Kok}
\affil{\small Tshwane Metropolitan Police Department\\ City of Tshwane, Republic of South Africa \\ {\tt kokkiek2@tshwane.gov.za.}}

\author{Naduvath Sudev}
\affil{\small Department of Mathematics\\ Vidya Academy of Science \& Technology \\ Thalakkottukara, Thrissur - 680501, India.\\ {\tt sudevnk@gmail.com}}

\date{}

\begin{document}
\maketitle

\begin{abstract}
Consider a network $D$ of pipes which have to be cleaned using some cleaning agents, called brushes, assigned to some vertices. The tattooing of a simple connected directed graph $D$ is a particular type of the cleaning in which an arc are coloured by the colour of the colour-brush transiting it and the tattoo number of $D$ is a corresponding derivative of brush numbers in it. In this paper, we introduce a new concept, called the tattoo index of a given graph $G$, which is an efficiency index related to the tattooing sequence and we establish some introductory results on this parameter.
\end{abstract}


\vspace{0.35cm}

\ni \textbf{Mathematics Subject Classification:} 05C20, 05C35, 05C38, 05C99.
 
\section{Introduction}

For a general reference to notation and concepts of graph theory and digraph theory, not defined specifically in this paper, please see \cite{BM1,BLS,CL1,FH,EWW,DBW}.  For the graph colouring concepts, please see \cite{CZ1,JT1}. Unless mentioned otherwise, all graphs and digraphs considered in this paper are simple, finite, connected and non-trivial.

\subsection{Graph Cleaning and Brush Number}

A graph cleaning model was introduced in \cite{SM1} as a combination of graph searching problems (see \cite{TDP1,TDP2}) and chip firing problems (see \cite{BLS1}). Assume that we have to clean a network $D$ of pipes on periodical basis using some cleaning agents called \textit{brushes}, assigned to some vertices of $D$. When a vertex is cleaned, a brush must travel down each contaminated edge. A contaminated edge can be considered to be cleaned when a brush traverse that edge. A graph $G$ is said to be cleaned when all its edges are cleaned.

\vspace{0.25cm}

In graph cleaning models, it is initially set that all edges of a simple connected undirected graph $G$ are dirty. We need to allocate a finite number of brushes, $\beta_G(v) \ge 0$ to each vertex $v \in V(G)$. The $\beta_G(v)$ brushes allocated to any vertex $v$  may clean the vertex and send one brush along each dirty edge and allocate an additional brush to its corresponding neighbour (vertex). When the edges of a simple connected graph $G$ are given some orientation so that $G$ becomes a directed graph then the brush allocated to a vertex may only clean along an out-arc from that vertex. The minimum number of brushes to be allocated to clean a graph for a given orientation $\alpha_i(G)$ is called the \textit{brush number} of the graph $G$, denoted $b_r^{\alpha_i}(G)$.

\subsection{Tattooing of Graphs}

As a special case of the graph cleaning models, the notion of tattooing of a graph was introduced in \cite{KS1}. In all studies on graph cleaning and brush number of graphs, brushes are considered to be identical objects and when cleaning is initiated from a vertex $v$, a particular brush can clean along any out-arc at random. The \textit{tattooing} of a simple connected directed graph $D$ is a particular type of the cleaning in which an arc will be coloured by the colour of the colour-brush transiting it and the \textit{tattoo number} of the directed graph $D$ is a corresponding derivative of brush numbers in it. 

\vspace{0.25cm}

Initial colours will be called \textit{primary colours} and primary colours will be allowed to blend into additional \textit{colour blends}. Consider a set of colour-brushes $\mathcal{C} = \{c_i:1\le i \le n; n \in \N\}$ of \textit{primary colours}. Let the initial allocation (at $t= 0$ of the tattooing process), say $s$, $s\ge 0$, of primary colours to a vertex $v$ of $G$ be the set $X_{t=0}(v) =\{c_i: i = 1,2,3,\ldots,s; s\in \N\}$. This allocation is allowed to mutate (which is also called blending) to include all possible \textit{primary blends} of colour-brushes. Hence, this allocation is the set $\cP_0(X_{t=0}(v))$ of all non-empty subsets of $(X_{t=0}(v))$.

\vspace{0.25cm}

A primary colour blend is the colour blend of at least two distinct primary colours. Primary colour blends may not mutate into \textit{secondary blends}. During the $i^{th}$-step of tattooing a number of identical primary colour-brushes or primary colour blends may arrive at a vertex $v$ without repetition. Obviously, $\{c_1,c_2,c_3\ldots, c_s\}_{s=0} = \emptyset$. 

\ni The following are some important definitions provided in \cite{KS1}.

\begin{defn}\label{Defn-2.1}{\rm 
At the $i^{th}$-step, $i\ge 0$, the \textit{tattoo power set} of $X_{t=i}(v)$ is defined to be
\begin{enumerate}
\item[(i)] $\cP^*(X_{t=i}(v)) = \cP_0(\{c_1,c_2,c_3,\ldots,c_s\})$ or $\emptyset$ (typically, but not exclusively at $0^{th}$-step), or
\item[(ii)] $\cP^*(X_{t=i}(v)) = \cP_0(\{c_i, c_j,c_k,\ldots c_t\})$, because only primary colour-brushes have arrived, or
\item[(iii)] $\cP^*(X_{t=i}(v)) = \cP_0(\{c_i, c_j,c_k,\ldots c_t\}\cup \{$primary colour blends $\in \cP^*_0(X_{t=i-1}(u)), \\ (u,v) \in A(G)\}$, or
\item[(iv)] $\cP^*(X_{t=i}(v)) = \{primary colour blends \in \cP^*_0(X_{t=i-1}(u)), (u,v) \in A(G)\}$ only.
\end{enumerate}
}\end{defn}

\begin{defn}\label{Defn-2.2}{\rm 
The \textit{tattoo label} of an arc $a_i$, denoted $l(a_i)$, is defined to be the ordered subscript(s) of either the primary colour-brush or the blended colour-brush tattooing along the arc. Furthermore, the sum of the entries of $l(a_i)$ is denoted $l_\Sigma(a_i)$.
}\end{defn}

\begin{defn}{\rm 
The \textit{random tattoo number} of a simple connected and randomly directed graph $G$ having orientation $\alpha_i(G)$, denoted by $\tau^{\alpha_i}(G)$, is the minimum number of times primary colour-brushes are allocated to vertices of $G$ to iteratively tattoo along all arcs of $G$ (excluding the transition of a primary colour-brush from one vertex to another). The \textit{tattoo number} of a simple connected graph $G$ denoted $\tau(G)$ is defined to be $\tau(G) = \min\{\tau^{\alpha_i}(G):\forall\, \alpha_i(G)\}\}$.
}\end{defn}

Motivated by various studies on brush number on graphs and the study on the tattoo number of graphs, in this paper, we now introduce certain new colouring parameter namely tattoo index of graphs and study some important and interesting properties of these parameters.  

\section{Tattoo Index of Graphs}

If a graph requires only the primary colour-brush $\{c_1\}$ to tattoo all arcs, the process is absolutely efficient. The more colour-brushes and colour blends needed, the less the tattooing efficiency will be. This can be expressed by the ratio $\frac{\epsilon(G)}{\tau(G)\cdot\sum\limits_{a_i\in A(G)}l_\Sigma(a_i)}$.

\vspace{0.25cm}

Let the initial colour-brush allocation of a directed cycle $C_n$ be $v_1 \mapsto\{c_1,c_2\},v_2\mapsto \emptyset, v_3\mapsto \emptyset \ldots v_n\mapsto \emptyset$. After mutation, we have $v_1 \mapsto\{\{c_1\},\{c_2\},\{c_1,c_2\}\},v_2\mapsto \emptyset, v_3\mapsto \emptyset \ldots v_n\mapsto \emptyset$. The tattooing is possible in one of the following six different ways.

\begin{enumerate}\itemsep0mm
\item[(a)] $c_1$ tattoo along $a_1,a_2,\ldots,a_{n-1}$ and $c_2$ tattoo along $(v_1,v_n)$ with $l(a_1) = l(a_2)=\ldots = l(a_n) = (1)$ and $l((v_1,v_n)) = (2)$ or vice versa, or 
\item[(b)] $c_1$ tattoo along $a_1,a_2,\ldots,a_{n-1}$ and $c_{1,2}$ tattoo along $(v_1,v_n)$ with $l(a_1) = l(a_2)=\ldots = l(a_n) = (1)$ and $l((v_1,v_n)) = (1,2)$  or vice versa, or
\item[(c)] $c_2$ tattoo along $a_1,a_2,\ldots,a_{n-1}$ and $c_{1,2}$ tattoo along $(v_1,v_n)$ with $l(a_1) = l(a_2)=\ldots = l(a_n) = (2)$ and $l((v_1,v_n)) = (1,2)$  or vice versa.
\end{enumerate}

Since $\cP^*_0(\{c_1\})= \{c_1\}$, the initial allocation of one colour-brush to a vertex will not be sufficient and therefore, we have $\tau(C_n)=2$. Therefore, it follows that for $\tau(C_7)=2$, the tattoo ratio ranges over $\frac{7}{16},\frac{7}{18},\frac{7}{26},\frac{7}{30},\frac{7}{38}$ and $\frac{7}{40}$. We observe that choice of arcs plays an important role here. The arc choice pattern may be denoted by $\mathfrak{A}_p(G)$. These observations motivate us to define a new parameter namely the tattoo index of a graph as follows.

\begin{defn}{\rm 
The \textit{tattoo index} of a graph $G$, denoted by $\mathfrak{T}(G)$, is defined as $\mathfrak{T}(G)=\max\{\frac{|E(G)|}{\tau(G)\cdot\sum\limits_{a_i\in A(G)}l_\Sigma(a_i)}\}$, where the maximum is taken over all $\mathfrak{A}_p(G)$ with respect to an optimal orientation.
}\end{defn}

To ensure minimality of arc labels (hence, maximality of the index) any addition of primary colours needed at a vertex $v$ to proceed with tattooing must be primary colours with smallest subscripts distinct from those already allocated to $v$. An analysis of the tattoo index of a graph is a complex problem. The brushing process introduced in  \cite{KSK1} is a special case of tattooing in that only one primary colour $c_1$ is utilised. Therefore, arc labels are always  absolute minima ensuring maximality of the tattooing index. 

\subsection{First Step Generalisation}

A first step towards generalisation would be to utilise distinct primary colours but to prohibit mutation (colour blending). Hence, if a vertex $v$ requires the allocation of say $\{\underbrace{c_1,c_1,\ldots,c_1}_{t-entries}\}$, for brush cleaning it is rather allocated $\{c_1,c_2,c_3,\ldots,c_t\}$. The minimum number of times colour-brushes are required by a graph $G$ with regards to this \textit{first step generalisation} abbreviated as \textit{FSG$_\tau(G)$} and is denoted by $b_\tau(G)$. Note that since repetition of identical colour-brushes at vertex $v$ during the $i^{th}$-step is eliminated in preparation of the $(i+1)$-th step of tattooing,  $b_r(G) \le b_\tau(G)$. An application is given for \textit{general friendship graphs} and the $J9$-graphs. First, we consider the conventional friendship graph $Fr(3,n)$.

\begin{prop}\label{Prop-3.1}
For friendship graphs in respect of FSG$(Fr(3,n))$, $n\ge 1$, we have
\begin{enumerate}\itemsep0mm
\item[(i)] $b_\tau(Fr(3,n))=2(n-1)$, 
\item[(ii)] $\mathfrak{T}_{FSG}(Fr(3,n))=\frac{3n}{2(n-1)(3n^2-5n +6)}$.
\end{enumerate}
\end{prop}
\begin{proof}
\textit{Part (i):} Without loss of generality, consider cycle $1$, cycle $2$, \ldots, cycle $n-2$. Allocate the set of primary colour-brushes $\{c_1,c_2\}$ to each vertex $v_{1,i}$, $1\le i \le n-2$. Clearly, tattooing may initiate and on completion thereof the colour-brushes $\{\underbrace{c_1,c_1,\ldots,c_1}_{(n-2) -entries},\underbrace{c_2,c_2,\ldots,c_2}_{(n-2) -entries}\} \mapsto \{c_1,c_2\}$ are available at the common vertex $u$ for the next tattooing step. Hence, two additional primary colour-brushes are required, therefore $\{c_1,c_2,c_3,c_4\}$ suffices to complete tattooing. The primary colour-brush count is $2(n-1)$. Had $\{c_1,c_2\}$ been allocated to the $(n-1)^{th}$ cycle as well, tattooing is feasible with a primary colour-brush count of $2(n-1)$. Had $\{c_1,c_2\}$ been allocated to the $n^{th}$ cycle tattooing is possible with a primary colour-brush count of $2n$. Clearly, the initial primary colour-brush allocation (not unique) is a minimum over all possible allocations. Hence, $b_{\tau}(Fr(3,n)) = 2(n-1)$.

\textit{Part (ii):} From the primary colour-brush allocation in Part-$1$ and the definition of $\mathfrak{T}(G)$, it clearly follows that 
$\mathfrak{T}_{FSG}(Fr(3,n)) = \frac{3n}{2(n-1)(4 +4(n-1)+ \sum\limits_{i=0}^{n-1}6i)}=\frac{3n}{2(n-1)(3n^2 -5n +6)}$.
\end{proof}

\begin{defn}{\rm 
A \textit{general friendship graph}, denoted by $Fr(n_i,m_j,s_k,\ldots,t_q);\ n,m,\\ s,\ldots,t \ge 3$; $i,j,k,\ldots, q \ge 1$, is the graph obtained by taking $i$ copies of $C_n$; $j$ copies of $C_m$; $k$ copies of $C_s$;\ldots; $q$ copies of $C_t$ and joining them at a common vertex.
}\end{defn}

\begin{prop}\label{Prop-3.2}
For a general friendship graph $Fr(n_i,m_j,\ldots,t_q)$, let $\kappa = i+j+\ldots+q$. Then, with respect to $(Fr(n_i,m_j,\ldots,t_q))$ we have
\begin{enumerate}
\item[(i)] $b_\tau(Fr(n_i,m_j,\ldots,t_q))=2(\kappa-1)$.
\item[(ii)] $\mathfrak{T}_{FSG}(Fr(n_i,m_j,\ldots,t_q)) =
\begin{cases}
\frac{ni+mj+\ldots+tq}{2(\kappa-1)(i(n+1)+j(m+1)+\ldots+(q-1)(t+1) + 10)}; & q\ge 2,\\
\frac{ni+mj+\ldots+tq}{2(\kappa-1)(i(n+1)+j(m+1)+\ldots+10)}; & q=1.
\end{cases}$
\end{enumerate}
\end{prop}
\begin{proof}
\textit{Part (i):} Let the common vertex be $u$ and label each cluster of cycles say, the cluster $s_k$ as cycle $C^{(1)}_s$, cycle $C^{(2)}_s$, \ldots, cycle $C^{(k)}_s$. Also, let any cycle $C^{(i)}_s \in s_k$, $1\le i \le k$ be $uv^{(s)}_{1,i}$$,v^{(s)}_{2,i}$, $v^{(s)}_{3,i}$, \ldots, $v^{(s)}_{s-1,i},u$. Since $\tau(C_n) =2$, $n\ge 3$ the initial allocation of $\{c_1,c_2\}$ sets at any $\kappa -2$ cycles could be at $v^{(s)}_{1,i}$, respectively. The proof then follows similar to that of Proposition \ref{Prop-3.1}(Part (i)).

\vspace{0.25cm}

\textit{Part (ii):} Assume without loss of generality, that $C_t$ are the smallest cycles and that $C_s$ are the second smallest cycles. Here, we need to consider the following cases.

\vspace{0.25cm}

\textit{Case (i):} If $q\ge 2$ then, since $\tau(C_n) =2,\ n\ge 3$, the initial allocation of $\{c_1,c_2\}$ sets at all $\kappa-2$ cycles excluding two copies of $C_q$, at $v^{(s)}_{1,i}$, $\forall\, i$ in $C_n,C_m,C_s,\ldots,C_t$ to ensure minimum sum of arc labels. Applying the \textit{counting technique} of Part (i), we have $\mathfrak{T}_{FSG}(Fr(n_i,m_j,s_k,\ldots,t_q)) =\frac{ni+mj+sk+\ldots +tq}{2(\kappa-1)(i(n+1)+j(m+1)+k(s+1)+\ldots+ (q-1)(t+1)+10)}$.

\vspace{0.25cm}

\textit{Case (ii):} If $q=1$ then allocate $\{c_1,c_2\}$ sets at all $\kappa -2$ cycles excluding $C_t$ and one copy of $C_s$, at $v^{(s)}_{1,i};\ \forall\, i$ in $C_n,C_m,C_s,\ldots,C_t$ to ensure minimum sum of arc labels. Clearly, the cycle of $C_t$ must be tattooed by $\{c_3,c_4\}$. Applying the \textit{counting technique} of Part (i), we have the equation as follows.
$$\mathfrak{T}_{FSG}(Fr(n_i,m_j,s_k,\ldots,t_q)) = \frac{ni+mj+sk+\ldots +tq}{2(\kappa-1)(i(n+1)+j(m+1)+k(s+1)+\ldots+ 10)}.$$
This completes the proof.
\end{proof}

The notion of a new family of graphs namely $J9$-graphs and the notion of a Joost graph has been introduced in \cite{KS1} as follows.  

\begin{defn}\label{Defn-2.5}{\rm 
Consider the inter-connected paths $u_1,v_{1,j},v_{2,j},\ldots,v_{n-2,j},u_2$, $n\ge 3$ for $j=1,2,3,\ldots,k$, where $k\ge 1$. The family of graphs is called \textit{$J9$-graphs}. A member of the $J9$-graphs is denoted $P^{(k)}_n$ and is called a \textit{Joost graph}.
}\end{defn}

\begin{prop}\label{Prop-3.3}
For $J9$-graphs in respect of $FSG(P^{(k)}_n)$, $n\ge 3$ we have
\begin{enumerate}
\item[(i)] $b_\tau(P^{(k)}_n) = k$, and
\item[(ii)] $\mathfrak{T}_{FSG}(P^{(k)}_n) =
\begin{cases}
1, & \text {if}\ k=1,\\
\frac{2(n-1)}{2n-1}, & \text{if}\ k=2,\\
\frac{2k(n-1)}{2n+ (n-1)k(k-1)}, & \text {if}\ k \ge 3.
\end{cases}$
\end{enumerate}
\end{prop}
\begin{proof}
\textit{Part (i):} Since $P^{(1)}_n = P_n$, it follows that $b_\tau(P^{(1)}_n)=1$. Also, since $P^{(2)}_n=C_n$, $b_\tau(P^{(2)}_n)=2$. Consider $P^{(3)}_n$ and allocate a minimum set of colour-brushes, $\{c_1,c_2\}$ at say $v_{1,1}$. The aforesaid allocation is possible without any loss of generality. Clearly, $c_1$ can tattoo along arc $(v_{1,1},u_1)$ and $c_2$ can tattoo along all arcs towards $u_2$, or vice versa. Till this point, the tattoo count is $2$. Consider a vertex, say $u_1$, with $\{c_2\}$ allocated to it. Since $d_{P^{(3)'}_n}(u_1)=d_{P^{(2)}_n}(u_1)=2$, a minimum number of additional primary colours must be added to have $\{c_1,c_2\}$ available to proceed with tattooing. Hence, the total tattoo count is $3$.

\vspace{0.25cm}

Assume that the result holds for all $1\le k\le \ell$ hence, $b_\tau(P^{(\ell)}_n) = \ell$. Replicating the colour-brush allocation procedure for $P^{(\ell)}_n$, which is iteratively similar to that described for $P^{(3)}_n$, it follows through immediate induction to hold for $P^{(\ell+1)}_n$. Therefore, the result holds in general.

\vspace{0.25cm}

\textit{Part (ii):} Cases (i), (ii) follow trivially.

\vspace{0.25cm}

\textit{Case (iii):} Let $k=3$. The colour-brush allocation procedure described in Part (i) ensures the minimum arc labeling. Then, we have 
\begin{eqnarray*}
\mathfrak{T}_{FSG}(P^{(k)}_n) & = & \frac{k(n-1)}{2+(n-2) + (n-1)\sum\limits_{i+1}^{k-1}i}\\
& = & \frac{k(n-1)}{2+(n-2) + (n-1)\sum\limits_{i+1}^{k-1}i}\\
& = & \frac{k(n-1)}{2+(n-2) + \frac{1}{2}(n-1)k(k-1)} \\
& = & \frac{2k(n-1)}{2n+ (n-1)k(k-1)}.
\end{eqnarray*}
This completes the proof.
\end{proof}

\subsection{Tattoo Index}

From the first generalisation step we now progress to the tattoo index. Consider the set $\{c_1,c_2,c_3,\ldots,c_k\}$. Order the subsets as follows.

\vspace{-0.65cm}

\begin{align*}
& \{c_1\},\{c_2\},c_3\},\ldots,\{c_k\},\\
& \{c_1,c_i\}_{2\le i\le k},\{c_2,c_i\}_{3\le i\le k},\{c_3,c_i\}_{4\le i\le k},\ldots,\{c_{k-1},c_k\},\\
& \{c_1,c_2,c_i\}_{3\le i\le k},\{c_1,c_3,c_i\}_{4\le i\le k},\{c_1,c_4,c_i\}_{5\le i\le k},\ldots,\{c_1,c_{k-1},c_k\},\\
& \ldots \ldots \ldots\\
& \ldots \ldots \ldots\\
& \{c_1,c_2,c_3,\ldots,c_k\}.
\end{align*}

The same ordering principle will apply to a set of colour-brushes $\{c_i,c_j,c_k,\ldots,c_\ell\}$, $i<j<j,\ldots<\ell$.

%%%%%%%%%%%%%%%%%%%%%%%%%%%%%%%%%
%%%%%%%%%%%%%%%%%%%%%%%%%%%%%%%%%%%%%%%%
%%%%%%%%%%%%%%%%%%%%%%%%%%%%%%%%%%%%%%%%%%%%%%%%%

From Definition \ref{Defn-2.2}, it follows that $l_\Sigma(a_i)$ is minimised by dispatching primary colour-brushes or colour blends sequentially in the order depicted above. Any $n\in \N$ lies in the interval $2^i\le n \le 2^{i+1}-1$ for some $i\in \N_0$. It implies that a vertex $v$ with $d^+(v) = t$ requires $\lceil log_2(t+1)\rceil$ primary colours allocated. Therefore, a closed formula does not exist to determine the arc labels of the arcs corresponding to $d^+(v)$. It must follow the \textit{rule of sequential colour-brush allocation.}

\begin{ill}{\rm 
For the friendship graph $Fr(3_6)$, the tattoo index $\mathfrak{T}(Fr(3_6))$ is determined by first tattooing a cycle $C_3$ with $\{c_1,c_2\}$. Following this first step, $5$ cycles $C_3$ remain with $\{c_1,c_2\}$ at the common vertex $u$. Thus, $10$ out-arcs await tattooing and hence $\lceil log_2(10+1)\rceil$ primary colours must be allocated. Hence, $\tau(Fr(3_6))=4$. By allocating the colour-brushes $\{c_1,c_2,c_3,c_4,c_{1,2},c_{1,3},c_{1,4},c_{2,3},c_{2,4},c_{3,4}\}$, tattooing can be completed. Therefore, the minimum sum of arc labels is given by $4+(4+10+10+16+19)=63$. Also, the number of arcs is $18$. Hence, $\mathfrak{T}(Fr(3_6)) = \frac{18}{4\times63}=\frac{1}{14}$.}
\end{ill}

\begin{ill}{\rm
For the Joost graph $P^{(7)}_4$, the the tattoo index $\mathfrak{T}(P^{(7)}_4)$ is determined by allocating $\{c_1,c_2,c_3,c_{1,2},c_{1,3},c_{2,3},c_{1,2,3}\}$ to $u_1$. The minimum sum of arc labels is given by $3+6+9+9+12+15+18 = 72$.  Since $\tau(P^{(7)}_4) =\lceil log_2(7+1)\rceil = 3$, $\mathfrak{T}(P^{(7)}_4) = \frac{21}{3\times72}=\frac{7}{72}$.}
\end{ill}

\section{Conclusion}

In this paper, we have discussed a new concept namely the tattoo index associated with tattooing, as a colouring extension of cleaning models used in graph theory and introduced some interesting concepts and parameters in that area. As pointed out in [10] the concept of tattooing has many real world applications and many of them are still to be explored and to be discovered. This concept is mainly based on the fact that many biological or virtual propagation models or mechanisms rely on mutation to reach a threshold level before propagation ignites. 

\vspace{0.25cm}

In \cite{KS1}, it is observed that for a friendship graph, the optimal tattooing sequence begin at a vertex with minimum degree and for a Joost graph, the optimal tattooing sequence begins at a vertex with maximum degree. In a wheel graph $W_{n+1}=C_n+K_1$, an optimal sequence must begin at a vertex of degree equal to $3$. Although the allocation $\{c_1,c_2,c_{1,2}\}$ at the vertex with degree equal to $3$, implies $\tau(W_{n+1})=2$, it does not provide the optimal tattoo index. These observations emphasise the inherent complexity of tattooing a graph optimally.

\vspace{0.25cm}

Besides determining the invariants $b_\tau(G), \mathfrak{T}_{FSG}(G)$ and $\tau(G)$ the invariant $\mathfrak{T}(G)$ is open for complexity and probability analysis. Determining the tattoo number and tattoo index of different graph classes offers much for further investigations. An algorithmic study of this graph parameter is also worth for future studies.


\begin{thebibliography}{25}

\bibitem{BLS1} A. Bj\"{o}rner, L. Lovasz, and P. Shor, \textit{Chip firing games on graphs}, European J. Combin. \textbf{12}(1991), 283-291.

\bibitem{BM1} J. A. Bondy and U.S.R. Murty, \textbf{Graph theory with applications,} Macmillan Press, London, 1976.

\bibitem {BLS} A. Brandst\"{a}dt, V. B. Le and J. P. Spinrad, {\bf Graph classes: A survey}, SIAM, Philadelphia, 1999.

\bibitem{CL1} G. Chartrand and L. Lesniak, \textbf{Graphs and digraphs}, CRC Press, 2000.

\bibitem{DD1} D. Dyer, \textbf{Sweeping graphs and digraphs}, Ph.D. thesis, Simon Fraser University, Canada, 2004.

\bibitem{CZ1} G. Chartrand and P. Zhang, \textbf{Chromatic graph theory}, CRC Press, 2009.

\bibitem {FH}  F. Harary, {\bf Graph theory}, Addison-Wesley Pub. Co. Inc., Philippines, 1969.

\bibitem{LHH1} L. H. Harper, \textit{Optimal assignments of numbers to vertices}, SIAM J. Appl. Math., \textbf{12} (1964), 131–135.
	
\bibitem{JT1} T. R. Jensen and B. Toft, {\bf Graph colouring problems}, John Wiley \& Sons, 1995.

\bibitem{KS1} J. Kok, and N. K. Sudev, \textit{Tattooing and the tattoo number of graphs}, Preprint,  arXiv:1603.00303.

\bibitem{KSK1} J. Kok, C. Susanth and S.J. Kalayathankal, \textit{Brush numbers of certain Mycielski graphs}, Int. J. Pure Appl. Math., {\bf 106}(2)(2016), 663-675., DOI: 10.12732/ijpam.v106i2.28.

\bibitem{SM1} S. McKeil, \textit{Chip firing cleaning process}, M. Sc. Thesis, Dalhousie University, Canada., 2007.

\bibitem{MEM1} M. E. Messinger, \textit{Methods of decontaminating a network}, Ph.D. Thesis, Dalhousie University, Canada., 2008.

\bibitem{MEM2} M. E. Messinger, R. J. Nowakowski and P. Pralat, \textit{Cleaning a network with brushes}. Theor. Comput. Sci., \textbf{399}(2008), 191-205.

\bibitem{TDP1} T. D. Parsons, \textit{Pursuit-evasion in a graph}, in \textbf{Theory and Applications of Graphs}, Y. Alavi and D. R. Lick, (Eds),  Springer, Berlin, 1976, 426-441.

\bibitem{TDP2} T. D. Parsons, \textit{The search number of a connected graph}, in \textbf{Proc. Ninth Southeastern Conf. Combinatorics, Graph Theory and Computing}, Congressus Numerantium, XXI, Winnipeg, 1978, 549-554.

\bibitem{TST1} T. S. Tan,\textit{The brush number of the two dimensional torus}, Pre-print, arXiv: 1012.4634v1.

\bibitem{EWW} E. W. Weisstein, {\bf CRC concise encyclopaedia of mathematics}, CRC press, 2011.

\bibitem {DBW} D. B. West, {\bf Introduction to graph theory}, Pearson Education Inc., 2001.

\end{thebibliography}
\end{document}